\documentclass{amsart}
\usepackage{tikz}
\usepackage{xcolor}
\usepackage{amssymb,latexsym,amsmath,extarrows}
\usepackage{graphicx,mathrsfs,comment}
\usepackage{hyperref,url}

\usepackage{amstext}
\usepackage{bbm}

\numberwithin{equation}{section}

\newtheorem{theorem}{Theorem}[section]
\newtheorem{lemma}[theorem]{Lemma}

\theoremstyle{remark}

\newcommand{\dint}{\displaystyle\int}

\newcommand{\R}{\mathbb{R}}

\newcommand{\1}{\mathbbm{1}}

\begin{document}

\title[]{On the density of rational lines on diagonal cubic hypersurfaces}

\author{Kiseok Yeon} \address{ Kiseok Yeon\\  Department of Mathematics, University of California, Davis, United States}\email{kyeon@ucdavis.edu}

\begin{abstract}

In this paper, we establish the asymptotic estimates for the rational lines on diagonal cubic hypersurfaces defined by $\sum_{i=1}^sc_ix^3_i=0$ with $c_i\in\mathbb{Z}\setminus \{0\},$ provided that $s\geq 19.$ This improves the previously known bound $s\geq 21$ required to obtain such asymptotic estimates. Our approach develops a multidimensional shifting variables argument together with a pruning argument, and exploits the recent progress on the Parsell-Vinogradov system.

\end{abstract}

\maketitle


\section{Introduction}
The problem of finding linear spaces on algebraic variety was initiated by Brauer \cite{MR13127} and Birch \cite{MR97359} in the middle of the last century as a key ingredient in their inductive approaches to establishing the existence of rational points. Quantitative developments of this problem have since been made by Brandes \cite{MR3198749},
Parsell \cite{MR1778504,MR1817503,MR2485413,MR2965969}, Wooley \cite{MR4744752} and Zhao \cite{MR3552298}; see also Parsell, Prendiville and Wooley \cite{MR3132907}. In \cite{MR1778504,MR1817503,MR2485413,MR2965969,MR3132907,MR4744752,MR3552298}, the authors investigate rational lines on diagonal hypersurfaces, while Brandes \cite[Theorem 2]{MR3198749} studies rational lines on algebraic variety defined by systems of homogeneous polynomial equations of the same degrees.

In this paper, we focus on rational lines on diagonal cubic hypersurfaces.  Consider 
\begin{equation}\label{cubic equation}
    \sum_{i=1}^sc_iz^3_i=0,
\end{equation}
where $c_i\ (1\leq i\leq s)$ are non-zero integers. 
We aim to count pairs of vectors $\boldsymbol{x}=(x_1,\ldots,x_s)\in \mathbb{Z}^s$ and $\boldsymbol{y}=(y_1,\ldots,y_s)\in \mathbb{Z}^s$ such that the line $l: \boldsymbol{x}+t\boldsymbol{y}$ is contained in the hypersurfaces defined by $(\ref{cubic equation}).$
One checks that such pairs $\boldsymbol{x}=(x_1,\ldots,x_s)\in \mathbb{Z}^s$ and $\boldsymbol{y}=(y_1,\ldots,y_s)\in \mathbb{Z}^s$ are counted by the system of equations
\begin{equation}\label{cubic system}
     \sum_{i=1}^sc_ix^3_i= \sum_{i=1}^sc_ix_i^2y_i= \sum_{i=1}^sc_ix_iy_i^2= \sum_{i=1}^sc_iy^3_i=0.
\end{equation}
Let $N_s(X):=N_s(X;\boldsymbol{c})$ denote the number of solutions $(\boldsymbol{x},\boldsymbol{y})\in \mathbb{Z}^{2s}\cap [-X,X]^{2s}$ to $(\ref{cubic system})$. Our main result is the following.

\begin{theorem}\label{theorem 1.1}
    Whenever $s\geq 19$, one has
    \begin{equation}\label{expected asymptotic formula}
        N_s(X)=\sigma X^{2s-12}+O(X^{2s-12-\delta}),
    \end{equation}
    for some $\delta>0$, where $\sigma:=\sigma_{\boldsymbol{c}}$ is a positive constant depending on the coeffcients $\boldsymbol{c}$.
\end{theorem}

Parsell \cite{MR2965969} previously established (\ref{expected asymptotic formula}) under the stronger condition $s\geq 29$. Earlier, Parsell \cite{MR1778504} deduced that $N_s(X)\gg P^{2s-12}$ for $s\geq 55$. Later, Zhao 
\cite{MR3552298} sharpened this to obtain $(\ref{expected asymptotic formula})$ for $s\geq 21,$ which has been the best available bound. Thus, Theorem \ref{theorem 1.1} represents an improvement over the bound $s\geq 21$ in \cite{MR3552298}.

In this paper, we use Vinogradov's notation, where $f\ll g$ means that $|f(x)|\leq Cg(x)$ for some sufficiently large constant $C>0.$ We may use $f=O(g)$ with the same meaning.

\bigskip

\section{Preliminary}
In this section, we will introduce some definitions and two lemmas from which Theorem \ref{theorem 1.1} immediately follows. Define
\begin{equation}\label{F_{cj}}
    F_{c_j}(\boldsymbol{\alpha}):=\sum_{|x|,|y|\leq X}e(c_j\alpha_1x^3+c_j\alpha_2x^2y+c_j\alpha_3xy^2+c_j\alpha_4y^3),
\end{equation}
where $\boldsymbol{\alpha}=(\alpha_1,\alpha_2,\alpha_3,\alpha_4)\in \mathbb{R}^4$. We see by the orthogonality that
\begin{equation*}
    N_s(X)=\int_{[0,1]^4} \prod_{j=1}^s F_{c_j}(\boldsymbol{\alpha})d\boldsymbol{\alpha}.
\end{equation*}
As in the previous results introduced in the first paragraph, we make use of the Hardy-Littlewood circle method in order to obtain $(\ref{expected asymptotic formula}).$ To do so, we must define
 the major and minor arcs. Define the major arcs $\mathfrak{N}_{\delta}$ by
\begin{equation*}
    \mathfrak{N}_{\delta}=\bigcup_{q\leq X^{\delta}}\bigcup_{\substack{1\leq \boldsymbol{a}\leq q\\ (q,\boldsymbol{a})=1}}\mathfrak{N}_{q,\boldsymbol{a}},
\end{equation*}
where
\begin{equation*}
    \mathfrak{N}_{q,\boldsymbol{a}}:=\{(\alpha_1,\alpha_2,\alpha_3,\alpha_4)\in [0,1)^4:\ |\alpha_i-a_i/q|\leq X^{\delta-3}\}.
\end{equation*}
Define $\mathfrak{n}_{\delta}:=[0,1)^4\setminus \mathfrak{N}_{\delta}.$ Here and throughout, we set $\delta=10^{-10}.$ By using these major and minor arcs dissections, we deduce that
\begin{equation}\label{major and minor arcs dissection}
    N_s(X)=\int_{\mathfrak{N}_{\delta}} \prod_{j=1}^s F_{c_j}(\boldsymbol{\alpha})d\boldsymbol{\alpha}+\int_{\mathfrak{n}_{\delta}} \prod_{j=1}^s F_{c_j}(\boldsymbol{\alpha})d\boldsymbol{\alpha}.
\end{equation}

The first lemma provides the asymptotic estimate for the first term on the right hand side in $(\ref{major and minor arcs dissection})$, which is already in the literature. To facilitate the statement of the lemma, we define
\begin{equation*}
S(q,\boldsymbol{a})=\sum_{x=1}^q\sum_{y=1}^qe\left(\frac{a_1x^3+a_2x^2y+a_3xy^2+a_4y^3}{q}\right)
\end{equation*}
and 
\begin{equation*}
    S(q)=\sum_{\substack{1\leq \boldsymbol{a}\leq q\\(q,\boldsymbol{a})=1}}\prod_{j=1}^s(q^{-2}S(q,c_j\boldsymbol{a})).
\end{equation*}
We define the singular series $\mathfrak{S}$ by
\begin{equation*}
    \mathfrak{S}=\sum_{q=1}^{\infty}S(q).
\end{equation*}
Furthermore, for $(\gamma_1,\gamma_2,\gamma_3,\gamma_4)\in \R^4,$ define
\begin{equation*}
u(\boldsymbol{\gamma})=\int_{-1}^1\int_{-1}^1e(\xi^3\gamma_1+\xi^2\eta\gamma_2+\xi\eta^2\gamma_3+\eta^3\gamma_4)d\xi d\eta.
\end{equation*}
We define the singular integral $\mathfrak{J}$ by
\begin{equation*}
\mathfrak{J}=\int_{\R^4}\prod_{j=1}^su(c_j\boldsymbol{\gamma})d\boldsymbol{\gamma}.
\end{equation*}

\begin{lemma}\label{lemma1.2}
Whenever $s\geq 16,$ one has
\begin{equation}\label{asymptotics for mean value over major arcs}
    \int_{\mathfrak{N}_{\delta}} \prod_{j=1}^s F_{c_j}(\boldsymbol{\alpha})d\boldsymbol{\alpha}=\mathfrak{S}\mathfrak{J}X^{2s-12}+O(X^{2s-12-\eta'}),
\end{equation}
for some $\eta'>0.$ Furthermore, we have  $\mathfrak{S}\asymp1$ and $\mathfrak{J}\asymp 1$.
\end{lemma}
\begin{proof}
 By \cite[section 7]{MR3552298}, one infers that whenever $s\geq 16,$ we have $(\ref{asymptotics for mean value over major arcs})$ and $\mathfrak{S},\mathfrak{J}\ll 1.$ Therefore, it suffices to show that $\mathfrak{S}\gg 1$ and $\mathfrak{J}\gg 1$, whenever $s\geq 16$. The claim  $\mathfrak{S}\gg1$ and $\mathfrak{J}\gg1$ can be verified by the standard argument \cite[section 6]{MR2485413}, under the hypothesis of non-singular local solutions, provided that $s\geq 16$. Meanwhile, the argument in \cite[Lemma 5.1]{MR1778504} yields the existence of non-singular solutions provided that $s\geq 14.$ Hence, we conclude that one has $\mathfrak{S}\gg 1$ and $\mathfrak{J}\gg 1$, whenever $s\geq 16.$
\end{proof}

\begin{lemma}\label{lemma1.3}
    Whenever $s\geq 19$, one has
\begin{equation*}
    \int_{\mathfrak{n}_{\delta}} \prod_{j=1}^s F_{c_j}(\boldsymbol{\alpha})d\boldsymbol{\alpha}\ll X^{2s-12-\eta''},
\end{equation*}    
for some $\eta''>0.$
\end{lemma}
We shall provide the proof of Lemma $\ref{lemma1.3}$ in section \ref{sec3}. By Lemma \ref{lemma1.2} and \ref{lemma1.3}, one infers from (\ref{major and minor arcs dissection}) that we have $(\ref{expected asymptotic formula})$, provided that $s\geq 19.$ This completes the proof of Theorem \ref{theorem 1.1}.

The main ingredients of the proof of Lemma $\ref{lemma1.3}$ are the sharp estimate for the cubic Parsell-Vinogradov's system \cite{MR3709122} and a generalization of the shifting variables argument originating from \cite{MR2913181}, together with application of a pruning argument. In section \ref{sec2}, we provide a multidimensional shifting variables argument as a generalization of shifting variables argument. In section \ref{sec3}, we provide the proof of Lemma \ref{lemma1.3} via Lemma \ref{lemma2.2} proved in section \ref{sec2}, together with a pruning argument.
We emphasize that applying the argument described in this paper, together with the sharp estimate for Parsell-Vinogradov's system \cite{MR3994585} and further generalized shifting variables argument, should improve the result of Parsell \cite[Theorem 1.2]{MR2485413}, giving the asymptotic bound for the number of linear spaces lying on diagonal hypersurfaces defined by 
$\sum_{i=1}^sc_iz_i^k=0,\ \text{with non-zero integers}\ c_i\ (1\leq i\leq s)$ and $k\geq 3.$

\bigskip
 
 \section{Multidimensional shifting variables argument}\label{sec2}


We begin this section by defining 
\begin{equation}\label{def1.1}    F(\boldsymbol{\alpha}):=\sum_{1\leq x,y\leq X}e(\alpha_1x^3+\alpha_2x^2y+\alpha_3xy^2+\alpha_4y^3),
\end{equation}
with $\boldsymbol{\alpha}=(\alpha_1,\alpha_2,\alpha_3,\alpha_4).$ 
In this section, we provide a key lemma associated with mean values of the exponential sum $F(\boldsymbol{\alpha})$, which plays a crucial role in  the proof of Lemma \ref{lemma1.3}.

We further define an auxiliary exponential sum $G(\boldsymbol{\alpha},\boldsymbol{\beta},\boldsymbol{\theta})$ with $\boldsymbol{\alpha}\in \R^4$, $\boldsymbol{\beta}\in \R^3$ and $\boldsymbol{\theta}\in \R^2$. In advance of defining this exponential sum, 
 we let $\nu_d(x,y):\R^2\rightarrow \R^N$ be the Veronese embedding with $N:=\binom{d+1}{d},$ defined by listing all the monomials of degree $d$ in two variables using the lexicographical ordering. Write $(\nu_{d}(x,y))_l$ for the $l$-th coordinate of $\nu_{d}(x,y).$ Define the exponential sum $G(\boldsymbol{\alpha}, \boldsymbol{\beta}, \boldsymbol{\theta})$ by
\begin{equation}\label{Galphabeta}
\begin{aligned}
G(\boldsymbol{\alpha},\boldsymbol{\beta},\boldsymbol{\theta})&:=G(\boldsymbol{\alpha},\boldsymbol{\beta},\boldsymbol{\theta};X)\\
&=\sum_{1\leq x,y\leq X}e\left(\sum_{l=1}^4\alpha_l(\nu_{3}(x,y))_l+\sum_{l=1}^3\beta_l(\nu_{2}(x,y))_l+\sum_{l=1}^2\theta_l(\nu_1(x,y))_l\right).
\end{aligned}
\end{equation}
Define
$J_s(X):= J_{s,2,3}(X)$ by
\begin{equation*}
   J_{s,2,3}(X)=\dint_{[0,1)^2}\dint_{[0,1)^3}\dint_{[0,1)^4}|G(\boldsymbol{\alpha},\boldsymbol{\beta},\boldsymbol{\theta};X)|^{2s}d\boldsymbol{\alpha}d\boldsymbol{\beta}d\boldsymbol{\theta}.
\end{equation*}
For a given $s\in \mathbb{N},$ define 
 \begin{equation*}
     \sigma_{s,d,l}(\boldsymbol{x},\boldsymbol{y})=\sum_{i=1}^s (\nu_{d}(x_i,y_i))_l-\sum_{i=s+1}^{2s} (\nu_{d}(x_i,y_i))_l.
 \end{equation*}
 For example, we see that
 \begin{equation*}
     \sigma_{s,3,2}(\boldsymbol{x},\boldsymbol{y})=\sum_{i=1}^s x_i^2y_i-\sum_{i=s+1}^{2s} x_i^2y_i
     \end{equation*}
     and
     \begin{equation*}
     \sigma_{s,2,3}(\boldsymbol{x},\boldsymbol{y})=\sum_{i=1}^s y_i^2-\sum_{i=s+1}^{2s} y_i^2.
 \end{equation*}


To facilitate the statement of the following lemma, we define the major arcs here by
\begin{equation*}
    \mathfrak{M}(H):=\bigcup_{\substack{1\leq a\leq q\leq H\\ (q,a)=1}}\mathfrak{M}_{q,a}(H),
\end{equation*}
where 
\begin{equation*}
    \mathfrak{M}_{q,a}(H):=\{\alpha\in [0,1):\ |\alpha-a/q|\leq q^{-1}HX^{-3}\}.
\end{equation*}
Define $\mathfrak{m}(H)=[0,1)\setminus \mathfrak{M}(H).$

\begin{lemma}\label{lemma2.2}
Let $H$ and $X$ be positive numbers with $H\leq X^{3/2}$. Suppose that $s$ is a natural number. Define
\begin{equation*}
    \begin{aligned}
        \mathcal{M}_1(H)&:=\mathfrak{m}(H)\times [0,1)\times [0,1)\times [0,1)\\
        \mathcal{M}_2(H)&:=[0,1)\times\mathfrak{m}(H)\times [0,1)\times [0,1)\\
        \mathcal{M}_3(H)&:=[0,1)\times[0,1)\times\mathfrak{m}(H)\times [0,1)\\
        \mathcal{M}_4(H)&:=[0,1)\times[0,1)\times[0,1)\times\mathfrak{m}(H).
    \end{aligned}
\end{equation*}
Then, for $1\leq l\leq 4,$ one has
   \begin{equation*}
       \begin{aligned}
\int_{\mathcal{M}_l(H)}|F(\boldsymbol{\alpha})|^{2s}d\alpha_1d\alpha_2d\alpha_3d\alpha_4\ll  X^{8} (\log X)^{4s}\cdot J_s(2X)\cdot (H^{-1}+X^{-1}).
       \end{aligned}
   \end{equation*} 
\end{lemma}

\begin{proof}

Recall the definition (\ref{Galphabeta}) of $G(\boldsymbol{\alpha,\boldsymbol{\beta}},\boldsymbol{\theta}).$
   First, we shall prove that for $1\leq l\leq 4,$ one has
    \begin{equation}\label{inequality2.1}
    \begin{aligned}
       & \int_{\mathcal{M}_l(H)}|F(\boldsymbol{\alpha})|^{2s}d\boldsymbol{\alpha} \\
        &\ll X^2 \sum_{\substack{|h_i|\leq sX^2\\ i=1,2,3}}\int_{[0,1)^5}\int_{\mathcal{M}_l(H)}|G(\boldsymbol{\alpha,\boldsymbol{\beta}},\boldsymbol{\theta})|^{2s}e(-\beta_1h_1-\beta_2h_2-\beta_3h_3)d\boldsymbol{\alpha} d\boldsymbol{\beta}d\boldsymbol{\theta}.
        \end{aligned}
    \end{equation}
Let us temporarily define an exponential sum $$\widetilde{F}(\boldsymbol{\alpha},\boldsymbol{\theta}):=\sum_{1\leq x,y\leq X}e(\alpha_1x^3+\alpha_2x^2y+\alpha_3xy^2+\alpha_4y^3+\theta_1x+\theta_2y).$$
Observe that when $(m_1,m_2)\in \mathbb{Z}^2$, one has
\begin{equation}\label{equation2.2}
\begin{aligned}
&\int_{[0,1)^2}\int_{\mathcal{M}_l(H)}|\widetilde{F}(\boldsymbol{\alpha},\boldsymbol{\theta})|^{2s}e(-m_1\theta_1-m_2\theta_2)d\boldsymbol{\alpha}d\boldsymbol{\theta}\\
&=\sum_{1\leq \boldsymbol{x},\boldsymbol{y}\leq X}\delta(\boldsymbol{x},\boldsymbol{y},\boldsymbol{m})\int_{\mathcal{M}_l(H)}e\biggl(\sum_{i=1}^4\alpha_i \sigma_{s,3,i}(\boldsymbol{x},\boldsymbol{y})\biggr)d\boldsymbol{\alpha},
\end{aligned}
\end{equation}
where
\begin{equation*}
\delta(\boldsymbol{x},\boldsymbol{y},\boldsymbol{m})=\left(\int_0^1e(\theta_1(\sigma_{s,1,1}(\boldsymbol{x},\boldsymbol{y})-m_1))d\theta_1\right)\left(\int_0^1e(\theta_2(\sigma_{s,1,2}(\boldsymbol{x},\boldsymbol{y})-m_2))d\theta_2\right)
\end{equation*}
in which 
\begin{equation*}
        \sigma_{s,1,1}(\boldsymbol{x},\boldsymbol{y})=\sum_{i=1}^s x_i-\sum_{i=s+1}^{2s} x_i,
        \end{equation*}
        and
        \begin{equation*}
        \sigma_{s,1,2}(\boldsymbol{x},\boldsymbol{y})=\sum_{i=1}^s y_i-\sum_{i=s+1}^{2s} y_i.
\end{equation*}
By orthogonality, for $i=1,2$, one has
\begin{equation*}
   \int_0^1e(\theta_i(\sigma_{s,1,i}(\boldsymbol{x},\boldsymbol{y})-m_i))d\theta_i=\left\{\begin{aligned}&1,\ \text{when}\ \sigma_{s,1,i}(\boldsymbol{x},\boldsymbol{y})=m_i\\
&0,\ \text{when}\ \sigma_{s,1,i}(\boldsymbol{x},\boldsymbol{y})\neq m_i.
   \end{aligned}\right.
\end{equation*}
When $1\leq \boldsymbol{x},\boldsymbol{y}\leq X$, one has $|\sigma_{s,1,i}(\boldsymbol{x},\boldsymbol{y})|\leq sX\ (i=1,2),$ and so 
\begin{equation*}
    \sum_{\substack{|m_i|\leq sX\\i=1,2}}\delta(\boldsymbol{x},\boldsymbol{y},\boldsymbol{m})=1.
\end{equation*}
Noting that 
\begin{equation*}
  \sum_{1\leq \boldsymbol{x},\boldsymbol{y}\leq X}e\biggl(\sum_{i=1}^4\alpha_i \sigma_{s,3,i}(\boldsymbol{x},\boldsymbol{y})\biggr)=|F(\boldsymbol{\alpha})|^{2s},
\end{equation*}
one deduces from $(\ref{equation2.2})$ that
\begin{equation} 
\begin{aligned}
&\sum_{\substack{|m_i|\leq sX\\i=1,2}}\int_{[0,1)^2}\int_{\mathcal{M}_l(H)}|\widetilde{F}(\boldsymbol{\alpha},\boldsymbol{\theta})|^{2s}e(-m_1\theta_1-m_2\theta_2)d\boldsymbol{\alpha}d\boldsymbol{\theta}\\
&=\dint_{\mathcal{M}_l(H)} \sum_{1\leq \boldsymbol{x},\boldsymbol{y}\leq X}\biggl( \sum_{\substack{|m_i|\leq sX\\i=1,2}}\delta(\boldsymbol{x},\boldsymbol{y},\boldsymbol{m})\biggr)e\biggl(\sum_{i=1}^4\alpha_i \sigma_{s,3,i}(\boldsymbol{x},\boldsymbol{y})\biggr)d\boldsymbol{\alpha}\\
&=\int_{\mathcal{M}_l(H)}|F(\boldsymbol{\alpha})|^{2s}d\boldsymbol{\alpha}.
\end{aligned}
\end{equation}
Hence, by applying the triangle inequality, we have
\begin{equation}\label{inequality2.4}
\int_{\mathcal{M}_l(H)}|F(\boldsymbol{\alpha})|^{2s}d\boldsymbol{\alpha}\ll X^2\int_{[0,1)^2}\int_{\mathcal{M}_l(H)}|\widetilde{F}(\boldsymbol{\alpha},\boldsymbol{\theta})|^{2s}d\boldsymbol{\alpha}d\boldsymbol{\theta}.
\end{equation}
Through a similar argument leading from $(\ref{equation2.2})$ to ($\ref{inequality2.4}$), we deduce that
\begin{equation}\label{equation2.5}
    \begin{aligned}
&\int_{[0,1)^2}\int_{\mathcal{M}_l(H)}|\widetilde{F}(\boldsymbol{\alpha},\boldsymbol{\theta})|^{2s}d\boldsymbol{\alpha}d\boldsymbol{\theta}\\
&=\sum_{\substack{|h_i|\leq sX^2\\ i=1,2,3}}\int_{[0,1)^5}\int_{\mathcal{M}_l(H)}|G(\boldsymbol{\alpha,\boldsymbol{\beta}},\boldsymbol{\theta})|^{2s}e(-\beta_1h_1-\beta_2h_2-\beta_3h_3)d\boldsymbol{\alpha} d\boldsymbol{\beta}d\boldsymbol{\theta}.    
    \end{aligned}
\end{equation}
Therefore, we confirmed the inequality $(\ref{inequality2.1})$, by substituting $(\ref{equation2.5})$ into the right hand side of $(\ref{inequality2.4}).$

Next, by shifting the variables, we have
\begin{equation}\label{equation2.6}
G(\boldsymbol{\alpha,\boldsymbol{\beta}},\boldsymbol{\theta})=\sum_{1+z_1\leq x\leq X+z_1}\sum_{1+z_2\leq y\leq X+z_2}e(\psi(x-z_1,y-z_2;\boldsymbol{\alpha,\boldsymbol{\beta}},\boldsymbol{\theta})),
\end{equation}
where
\begin{equation*}
\psi(x,y;\boldsymbol{\alpha,\boldsymbol{\beta}},\boldsymbol{\theta})=\sum_{l=1}^4\alpha_l(\nu_{3}(x,y))_l+\sum_{l=1}^3\beta_l(\nu_{2}(x,y))_l+\sum_{l=1}^2\theta_l(\nu_1(x,y))_l.
\end{equation*}
Write 
\begin{equation*}
    K(\gamma_1,\gamma_2)=\sum_{1\leq y_1,y_2\leq X}e(-\gamma_1y_1-\gamma_2y_2).
\end{equation*}
Then, we deduce from $(\ref{equation2.6})$ that when $1\leq z_1,z_2\leq X$, one has
\begin{equation}\label{equation2.7}
G(\boldsymbol{\alpha,\boldsymbol{\beta}},\boldsymbol{\theta})=\int_{[0,1)^2}F_{z_1,z_2}(\boldsymbol{\alpha,\boldsymbol{\beta}},\boldsymbol{\theta};\gamma_1,\gamma_2)K(\gamma_1,\gamma_2)d\gamma_1 d\gamma_2,
\end{equation}
where 
\begin{equation}\label{def f_z_1z_2}
    F_{z_1,z_2}(\boldsymbol{\alpha,\boldsymbol{\beta}},\boldsymbol{\theta};\gamma_1,\gamma_2)=\sum_{1\leq x,y\leq 2X}e(\psi(x-z_1,y-z_2;\boldsymbol{\alpha,\boldsymbol{\beta}},\boldsymbol{\theta})+\gamma_1(x-z_1)+\gamma_2(y-z_2)).
\end{equation}

Define 
\begin{equation}\label{def F_z_1z_2}
\begin{aligned}
   & \mathcal{F}_{z_1,z_2}(\boldsymbol{\alpha,\boldsymbol{\beta}},\boldsymbol{\theta};\boldsymbol{\gamma}_1,\boldsymbol{\gamma}_2)\\
    &=\prod_{i=1}^s F_{z_1,z_2}(\boldsymbol{\alpha,\boldsymbol{\beta}},\boldsymbol{\theta};\gamma_1^{(i)},\gamma_2^{(i)}) F_{z_1,z_2}(-\boldsymbol{\alpha,-\boldsymbol{\beta}},-\boldsymbol{\theta};-\gamma_1^{(s+i)},-\gamma_2^{(s+i)}).
\end{aligned}
\end{equation}
On substituting $(\ref{equation2.7})$ into $(\ref{equation2.5})$, we deduce that when $1\leq z_1,z_2\leq X$, one has
\begin{equation}\label{2.122.12}
    \begin{aligned}
&\int_{[0,1)^2}\int_{\mathcal{M}_l(H)}|\widetilde{F}(\boldsymbol{\alpha},\boldsymbol{\theta})|^{2s}d\boldsymbol{\theta}\\
& =\sum_{\substack{|h_i|\leq sX^2\\ i=1,2,3}}\int_{[0,1)^{2s}}\int_{[0,1)^{2s}}I_{h_1,h_2,h_3}(\boldsymbol{\gamma}_1,\boldsymbol{\gamma}_2,z_1,z_2)\widetilde{K}(\boldsymbol{\gamma}_1,\boldsymbol{\gamma}_2)d\boldsymbol{\gamma}_1d\boldsymbol{\gamma}_2,
    \end{aligned}
\end{equation}
where
\begin{equation}\label{2.112.11}
\begin{aligned}
    &I_{h_1,h_2,h_3}(\boldsymbol{\gamma}_1,\boldsymbol{\gamma}_2,z_1,z_2)\\
&=\int_{[0,1)^2}\int_{[0,1)^3}\int_{\mathcal{M}_l(H)}\mathcal{F}_{z_1,z_2}(\boldsymbol{\alpha,\boldsymbol{\beta}},\boldsymbol{\theta};\boldsymbol{\gamma}_1,\boldsymbol{\gamma}_2)e(-\beta_1h_1-\beta_2h_2-\beta_3h_3)d\boldsymbol{\alpha} d\boldsymbol{\beta}d\boldsymbol{\theta}
\end{aligned}
\end{equation}
and 
\begin{equation*}
    \widetilde{K}(\boldsymbol{\gamma}_1,\boldsymbol{\gamma}_2)=\prod_{i=1}^sK(\gamma_1^{(i)},\gamma_2^{(i)})K(-\gamma_1^{(s+i)},-\gamma_2^{(s+i)}).
\end{equation*}

We temporarily define 
\begin{equation*}
   \phi(\boldsymbol{x}-z_1,\boldsymbol{\gamma}_1)=\sum_{i=1}^s\gamma_1^{(i)}(x_i-z_1)-\sum_{i=1}^{s}\gamma_1^{(s+i)}(x_{s+i}-z_1)
\end{equation*}
and
\begin{equation*}
   \phi(\boldsymbol{y}-z_2,\boldsymbol{\gamma}_2)=\sum_{i=1}^s\gamma_2^{(i)}(y_i-z_2)-\sum_{i=1}^{s}\gamma_2^{(s+i)}(y_{s+i}-z_2).
\end{equation*}
Then, it follows by orthogonality that one has
\begin{equation}\label{2.92.92.9}
\begin{aligned}
&\int_{[0,1)^2}\int_{[0,1)^3}\mathcal{F}_{z_1,z_2}(\boldsymbol{\alpha,\boldsymbol{\beta}},\boldsymbol{\theta};\boldsymbol{\gamma}_1,\boldsymbol{\gamma}_2)e(-\beta_1h_1-\beta_2h_2-\beta_3h_3) d\boldsymbol{\beta}d\boldsymbol{\theta}\\
&=\sum_{1\leq \boldsymbol{x},\boldsymbol{y}\leq 2X}\Delta(\boldsymbol{\alpha},\boldsymbol{\gamma}_1,\boldsymbol{\gamma}_2,h_1,h_2,h_3,z_1,z_2),
\end{aligned}
\end{equation}
where 
\begin{equation*}
\begin{aligned}
&\Delta(\boldsymbol{\alpha},\boldsymbol{\gamma}_1,\boldsymbol{\gamma}_2,h_1,h_2,h_3,z_1,z_2)\\
&=e\biggl(\sum_{i=1}^4\alpha_i\sigma_{s,3,i}(\boldsymbol{x}-z_1,\boldsymbol{y}-z_2)+\phi(\boldsymbol{x}-z_1,\boldsymbol{\gamma}_1)+\phi(\boldsymbol{y}-z_2,\boldsymbol{\gamma}_2)\biggr),
\end{aligned}
\end{equation*}
when
\begin{equation}\label{2.9}
    \begin{aligned}
        &\sigma_{s,2,i}(\boldsymbol{x}-z_1,\boldsymbol{y}-z_2)=h_i\ (i=1,2,3)\\
        &\sigma_{s,1,i}(\boldsymbol{x}-z_1,\boldsymbol{y}-z_2)=0\ \ \ (i=1,2),
    \end{aligned}
\end{equation}
and otherwise $\Delta(\boldsymbol{\alpha},\boldsymbol{\gamma}_1,\boldsymbol{\gamma}_2,h_1,h_2,h_3,z_1,z_2)$ is equal to $0.$

By applying the Binomial Theorem, one sees that whenever $(\ref{2.9})$ is satisfied for $2s$-tuples $\boldsymbol{x}$ and $\boldsymbol{y}$, we have
    \begin{equation}\label{2.10}
    \begin{aligned}
        &\sigma_{s,2,i}(\boldsymbol{x},\boldsymbol{y})=h_i\ (i=1,2,3)\\
        &\sigma_{s,1,i}(\boldsymbol{x},\boldsymbol{y})=0\ \ \ (i=1,2),
    \end{aligned}
\end{equation}
and hence
\begin{equation*}
    \begin{aligned}
        &\sum_{i=1}^4\alpha_i\sigma_{s,3,i}(\boldsymbol{x}-z_1,\boldsymbol{y}-z_2)\\
        &=\sum_{i=1}^4\alpha_i\sigma_{s,3,i}(\boldsymbol{x},\boldsymbol{y})-3h_1z_1\alpha_1-(h_1z_2+2h_2z_1)\alpha_2-(h_3z_1+2h_2z_2)\alpha_3-3h_3z_2\alpha_4,
    \end{aligned}
\end{equation*}
where we have used the Binomial Theorem in the last equality.

Therefore, on writing that
\begin{equation}\label{def L_1L_2}
\begin{aligned}
   L_1(\boldsymbol{h},\boldsymbol{\alpha})&:= 3h_1\alpha_1+2h_2\alpha_2+h_3\alpha_3,\\
  L_2(\boldsymbol{h},\boldsymbol{\alpha})&:= h_1\alpha_2+2h_2\alpha_3+3h_3\alpha_4
\end{aligned}
\end{equation}
and
\begin{equation}\label{def L_3}
L_3(\boldsymbol{h},\boldsymbol{\beta}):=\beta_1h_1+\beta_2h_2+\beta_3h_3,
\end{equation} 
it follows from $(\ref{2.92.92.9})$ that
\begin{equation*}
    \begin{aligned}
&\int_{[0,1)^2}\int_{[0,1)^3}\mathcal{F}_{z_1,z_2}(\boldsymbol{\alpha,\boldsymbol{\beta}},\boldsymbol{\theta};\boldsymbol{\gamma}_1,\boldsymbol{\gamma}_2)e(-\beta_1h_1-\beta_2h_2-\beta_3h_3) d\boldsymbol{\beta}d\boldsymbol{\theta}  \\
&=w_{\boldsymbol{\gamma}_1,\boldsymbol{\gamma}_2,z_1,z_2}\int_{[0,1)^2}\int_{[0,1)^3}\mathcal{F}_{0,0}(\boldsymbol{\alpha,\boldsymbol{\beta}},\boldsymbol{\theta};\boldsymbol{\gamma}_1,\boldsymbol{\gamma}_2)\\
&\ \ \ \ \ \ \ \ \ \ \ \ \ \ \ \ \ \ \ \ \ \ \ \ \ \ \ \ \ \ \ \ \ \ \ \ \ \ \  \ \ \ \ \ \ \cdot e(-z_1L_1(\boldsymbol{h},\boldsymbol{\alpha})-z_2L_2(\boldsymbol{h},\boldsymbol{\alpha})-L_3(\boldsymbol{h},\boldsymbol{\beta})) d\boldsymbol{\beta}d\boldsymbol{\theta},
    \end{aligned}
\end{equation*}
where 
\begin{equation*}
w_{\boldsymbol{\gamma}_1,\boldsymbol{\gamma}_2,z_1,z_2}=e\biggl(-z_1\biggl(\sum_{i=1}^s\gamma_1^{(i)}-\sum_{i=1}^s\gamma_{1}^{(s+i)}\biggr)-z_2\biggl(\sum_{i=1}^s\gamma_2^{(i)}-\sum_{i=1}^s\gamma_{2}^{(s+i)}\biggr)\biggr).
\end{equation*}
From this, we deduce from $(\ref{2.112.11})$ that
\begin{equation}\label{2.132.13}
    \begin{aligned}
       &\sum_{\substack{|h_i|\leq sX^2\\ i=1,2,3}} I_{h_1,h_2,h_3}(\boldsymbol{\gamma}_1,\boldsymbol{\gamma}_2,z_1,z_2)\\
       &=w_{\boldsymbol{\gamma}_1,\boldsymbol{\gamma}_2,z_1,z_2}\int_{[0,1)^2}\int_{[0,1)^3}\int_{\mathcal{M}_l(H)}\mathcal{F}_{0,0}(\boldsymbol{\alpha,\boldsymbol{\beta}},\boldsymbol{\theta};\boldsymbol{\gamma}_1,\boldsymbol{\gamma}_2)\\
&\ \  \ \ \ \ \ \ \ \ \ \ \ \ \ \ \ \ \ \ \  \ \ \ \cdot \sum_{\substack{|h_i|\leq sX^2\\ i=1,2,3}}e(-z_1L_1(\boldsymbol{h},\boldsymbol{\alpha})-z_2L_2(\boldsymbol{h},\boldsymbol{\alpha})-L_3(\boldsymbol{h},\boldsymbol{\beta})) d\boldsymbol{\alpha}d\boldsymbol{\beta}d\boldsymbol{\theta}\\
&\ll \int_{[0,1)^2}\int_{[0,1)^3}\int_{\mathcal{M}_l(H)}|\mathcal{F}_{0,0}(\boldsymbol{\alpha,\boldsymbol{\beta}},\boldsymbol{\theta};\boldsymbol{\gamma}_1,\boldsymbol{\gamma}_2)|\\
&\ \  \ \ \ \ \ \ \ \ \ \ \ \ \ \ \ \ \ \ \  \ \ \cdot \biggl|\sum_{\substack{|h_i|\leq sX^2\\ i=1,2,3}}e(-z_1L_1(\boldsymbol{h},\boldsymbol{\alpha})-z_2L_2(\boldsymbol{h},\boldsymbol{\alpha})-L_3(\boldsymbol{h},\boldsymbol{\beta})) \biggr|d\boldsymbol{\alpha}d\boldsymbol{\beta}d\boldsymbol{\theta}.
\end{aligned}
\end{equation}

For simplicity, we temporarily define the exponential sum $T(\boldsymbol{\alpha},\boldsymbol{\beta}):=T(\boldsymbol{\alpha},\boldsymbol{\beta};X,Y)$ by
\begin{equation*}
\begin{aligned}
T(\boldsymbol{\alpha},\boldsymbol{\beta};X,Y):=\sum_{1\leq z_1,z_2\leq X}\biggl|\sum_{\substack{|h_i|\leq Y\\ i=1,2,3}}e(-z_1L_1(\boldsymbol{h},\boldsymbol{\alpha})-z_2L_2(\boldsymbol{h},\boldsymbol{\alpha})-L_3(\boldsymbol{h},\boldsymbol{\beta}))\biggr|.
\end{aligned}
\end{equation*} 
Then, we conclude from $(\ref{2.132.13})$ that 
\begin{equation*}
\begin{aligned}
    &\sum_{ 1\leq z_1,z_2\leq X}\sum_{\substack{|h_i|\leq sX^2\\ i=1,2,3}} I_{h_1,h_2,h_3}(\boldsymbol{\gamma}_1,\boldsymbol{\gamma}_2,z_1,z_2)\\
    &\ll \int_{[0,1)^2}\int_{[0,1)^3}\int_{\mathcal{M}_l(H)}|\mathcal{F}_{0,0}(\boldsymbol{\alpha,\boldsymbol{\beta}},\boldsymbol{\theta};\boldsymbol{\gamma}_1,\boldsymbol{\gamma}_2)|\cdot |T(\boldsymbol{\alpha},\boldsymbol{\beta};X,sX^2)|d\boldsymbol{\alpha}d\boldsymbol{\beta}d\boldsymbol{\theta}\\
    &\leq \int_{[0,1)^9}|\mathcal{F}_{0,0}(\boldsymbol{\alpha,\boldsymbol{\beta}},\boldsymbol{\theta};\boldsymbol{\gamma}_1,\boldsymbol{\gamma}_2)|d\boldsymbol{\alpha}d\boldsymbol{\beta}d\boldsymbol{\theta}\cdot \sup_{\substack{\boldsymbol{\alpha}\in \mathcal{M}_l(H)\\ \boldsymbol{\beta}\in [0,1)^3 }}|T(\boldsymbol{\alpha},\boldsymbol{\beta};X,sX^2)|.
    \end{aligned}
\end{equation*}

Recall the definitions $(\ref{def f_z_1z_2})$ and $(\ref{def F_z_1z_2})$. Then, we deduce by applying the H\"older's inequality that
\begin{equation}\label{almost final}
    \begin{aligned}
        &\sum_{1\leq z_1,z_2\leq X}\sum_{\substack{|h_i|\leq sX^2\\ i=1,2,3}} I_{h_1,h_2,h_3}(\boldsymbol{\gamma}_1,\boldsymbol{\gamma}_2,z_1,z_2)\\
&\leq \prod_{i=1}^{2s}\biggl(\int_{[0,1]^{9}}|F_{0,0}(\boldsymbol{\alpha,\boldsymbol{\beta}},\boldsymbol{\theta};\gamma_1^{(i)},\gamma_2^{(i)})|^{2s}d\boldsymbol{\alpha}d\boldsymbol{\beta}d\boldsymbol{\theta}\biggr)^{1/2s}\cdot \sup_{\substack{\boldsymbol{\alpha}\in \mathcal{M}_l(H)\\ \boldsymbol{\beta}\in [0,1)^3 }}|T(\boldsymbol{\alpha},\boldsymbol{\beta};X,sX^2)|\\
&\leq \sup_{(\gamma_1,\gamma_2)\in[0,1)^2}\int_{[0,1]^{9}}|F_{0,0}(\boldsymbol{\alpha,\boldsymbol{\beta}},\boldsymbol{\theta};\gamma_1,\gamma_2)|^{2s}d\boldsymbol{\alpha}d\boldsymbol{\beta}d\boldsymbol{\theta}\cdot \sup_{\substack{\boldsymbol{\alpha}\in \mathcal{M}_l(H)\\ \boldsymbol{\beta}\in [0,1)^3 }}|T(\boldsymbol{\alpha},\boldsymbol{\beta};X,sX^2)|\\
&\leq J_{s}(2X)\cdot \sup_{\substack{\boldsymbol{\alpha}\in \mathcal{M}_l(H)\\ \boldsymbol{\beta}\in [0,1)^3 }}|T(\boldsymbol{\alpha},\boldsymbol{\beta};X,sX^2)|.
    \end{aligned}
\end{equation}
Therefore, we conclude from $(\ref{inequality2.4}),(\ref{2.122.12})$ and $(\ref{almost final})$ that 
\begin{equation}\label{final}
    \begin{aligned}
&\int_{\mathcal{M}_l(H)}|F(\boldsymbol{\alpha})|^{2s}d\boldsymbol{\alpha}\\
&\ll X^2\int_{[0,1)^2}\int_{\mathcal{M}_l(H)}|\widetilde{F}(\boldsymbol{\alpha},\boldsymbol{\theta})|^{2s}d\boldsymbol{\theta}\\
&\leq \sum_{1\leq z_1,z_2\leq X}\sum_{\substack{|h_i|\leq sX^2\\ i=1,2,3}}\int_{[0,1)^{2s}}\int_{[0,1)^{2s}}I_{h_1,h_2,h_3}(\boldsymbol{\gamma}_1,\boldsymbol{\gamma}_2,z_1,z_2)\widetilde{K}(\boldsymbol{\gamma}_1,\boldsymbol{\gamma}_2)d\boldsymbol{\gamma}_1d\boldsymbol{\gamma}_2\\
&= J_{s}(2X)\cdot \int_{[0,1)^{2s}}\int_{[0,1)^{2s}}\widetilde{K}(\boldsymbol{\gamma}_1,\boldsymbol{\gamma}_2)d\boldsymbol{\gamma}_1d\boldsymbol{\gamma}_2\cdot \sup_{\substack{\boldsymbol{\alpha}\in \mathcal{M}_l(H)\\ \boldsymbol{\beta}\in [0,1)^3 }}|T(\boldsymbol{\alpha},\boldsymbol{\beta};X,sX^2)|.
    \end{aligned}
\end{equation}

Meanwhile, one has
\begin{equation*}
\begin{aligned}
    \int_{[0,1)^2} |K(\gamma_1,\gamma_2)|d\gamma_1d\gamma_2&\ll \int_{[0,1]^2}\min\{X,\|\gamma_1\|^{-1}\}\cdot\min\{X,\|\gamma_2\|^{-1}\}d\gamma_1d\gamma_2\\
    &\ll (\log X)^2,
\end{aligned}
\end{equation*}
and hence
\begin{equation}\label{K(gamma_1,gamma_2)}
\int_{[0,1)^{2s}}\int_{[0,1)^{2s}}\widetilde{K}(\boldsymbol{\gamma}_1,\boldsymbol{\gamma}_2)d\boldsymbol{\gamma}_1d\boldsymbol{\gamma}_2\ll (\log X)^{4s}.
\end{equation}
Additionally, one deduces that
\begin{equation}\label{bound for T}
\begin{aligned}
    &\sup_{\substack{\boldsymbol{\alpha}\in \mathcal{M}_2(H)\\ \boldsymbol{\beta}\in [0,1)^3 }}|T(\boldsymbol{\alpha},\boldsymbol{\beta};X,sX^2)|\\
    &\ll X^5\sup_{\substack{\boldsymbol{\alpha}\in \mathcal{M}_2(H)\\ \boldsymbol{\beta}\in [0,1)^3\\ 1\leq z_2\leq X }} \sum_{1\leq z_1\leq X}\biggl|\sum_{|h_3|\leq sX^2}e(-z_1h_3\alpha_3-3z_2h_3\alpha_4-\beta_3h_3)\biggr|\\
    &\ll X^5\sup_{\substack{\boldsymbol{\alpha}\in \mathcal{M}_2(H)\\ \boldsymbol{\beta}\in [0,1)^3 \\1\leq z_2\leq X}}\sum_{1\leq z_1\leq X}\text{min}\biggl(X^2,\frac{1}{\|z_1\alpha_3+3z_2\alpha_4+\beta_3\|}\biggr).
\end{aligned}
\end{equation}

Suppose that $|\alpha_3-b/r|\leq r^{-2}$ with $r\in \mathbb{N}$, $b\in \mathbb{Z}$ and $(r,b)=1$. Then, it follows by \cite[Lemma 3.2]{MR865981} that 
\begin{equation}\label{asdf}
    \sum_{1\leq z_1\leq X}\text{min}\biggl(X^2,\frac{1}{\|z_1\alpha_3+3z_2\alpha_4+\beta_3\|}\biggr)\ll X^{3+\epsilon}\biggl(\frac{1}{r}+\frac{1}{X}+\frac{r}{X^3}\biggr).
\end{equation}
Meanwhile, it follows by Dirichlet's approximation theorem that for a given $H$ with $H\leq X^{3}$, there exist $r\in \mathbb{N}$ and $b\in \mathbb{Z}$ such that $|\alpha_3-b/r|\leq r^{-1}HX^{-3}$ and $r\leq H^{-1}X^3.$ The fact that $\alpha_3\in \mathfrak{m}(H)$ ensures that $r>H.$ Therefore, we find by $(\ref{asdf})$ that 
\begin{equation}\label{almost last}
      \sum_{1\leq z_1\leq X}\text{min}\biggl(X^2,\frac{1}{\|z_1\alpha_3+3z_2\alpha_4+\beta_3\|}\biggr)\ll X^{3+\epsilon}(H^{-1}+X^{-1}).
\end{equation}

By substituting $(\ref{almost last})$ into $(\ref{bound for T})$, we see that  one has
\begin{equation}\label{bound for T2}
    \sup_{\substack{\boldsymbol{\alpha}\in \mathcal{M}_2(H)\\ \boldsymbol{\beta}\in [0,1)^3 }}|T(\boldsymbol{\alpha},\boldsymbol{\beta};X,sX^2)|\ll X^{8+\epsilon}(H^{-1}+X^{-1}).
\end{equation}
When $\mathcal{M}_1(H)$, $\mathcal{M}_3(H)$, $\mathcal{M}_4(H)$, one may choose $(z_2,h_3)$,  $(z_2,h_1)$, ($z_1,h_1$)  in the second expression in $(\ref{bound for T})$, instead of $(z_1,h_3)$ over which the exponential sum runs. Then, one infers by the argument leading from $(\ref{bound for T})$ to $(\ref{bound for T2})$ that for $l=1,2,4$, one has
\begin{equation}\label{bound for T3}
    \sup_{\substack{\boldsymbol{\alpha}\in \mathcal{M}_l(H)\\ \boldsymbol{\beta}\in [0,1)^3 }}|T(\boldsymbol{\alpha},\boldsymbol{\beta};X,sX^2)|\ll X^{8+\epsilon}(H^{-1}+X^{-1}).
\end{equation}

Substituting $(\ref{K(gamma_1,gamma_2)})$, $(\ref{bound for T2})$ and $(\ref{bound for T3})$ into the last expression in $(\ref{final})$, we complete the proof of Lemma $\ref{lemma2.2}.$
\end{proof}

\bigskip

\section{Proof of Lemma 2.2}\label{sec3}

Recall the definition $(\ref{def1.1})$ of $F(\boldsymbol{\alpha})$.
\begin{lemma}\label{prop3.1}
    Suppose that $\mathfrak{D}$ is a measurable set in $[0,1)^2$. Then, we have
    \begin{equation}
        \int_0^1\int_{\mathfrak{D}}\int_0^1 |F(\boldsymbol{\alpha})|^8d\boldsymbol{\alpha}\ll X^{10+\epsilon}\cdot \text{mes}(\mathfrak{D}),
    \end{equation}
    where $d\boldsymbol{\alpha}=d\alpha_1 d\alpha_2 d\alpha_3 d\alpha_4.$

\end{lemma}
\begin{proof}
    By orthogonality, we see that 
    \begin{equation}\label{first eq in prop3.1}
    \begin{aligned}
        \int_0^1\int_{\mathfrak{D}}\int_0^1 |F(\boldsymbol{\alpha})|^8d\boldsymbol{\alpha}&=\int_{\mathfrak{D}}\int_0^1\int_0^1 |F(\boldsymbol{\alpha})|^8d\alpha_1d\alpha_4d\alpha_2d\alpha_3\\
        &=\int_{\mathfrak{D}}\sum_{1\leq \boldsymbol{x},\boldsymbol{y}\leq X}\Delta(\boldsymbol{x},\boldsymbol{y},\alpha_3,\alpha_2)d\alpha_2d\alpha_3,
    \end{aligned}
    \end{equation}
where 
\begin{equation*}
\Delta(\boldsymbol{x},\boldsymbol{y},\alpha_3,\alpha_2)=  e\biggl(\alpha_2 \sigma_{4,3,2}(\boldsymbol{x},\boldsymbol{y})+\alpha_2 \sigma_{4,3,3}(\boldsymbol{x},\boldsymbol{y})\biggr),
\end{equation*}
when
\begin{equation}\label{hua's lemma}
    \begin{aligned}
        \sigma_{4,3,1}(\boldsymbol{x},\boldsymbol{y})&=\sum_{i=1}^4x_i^3-\sum_{i=5}^8x_{i}^3=0\\
         \sigma_{4,3,4}(\boldsymbol{x},\boldsymbol{y})&=\sum_{i=1}^4y_i^3-\sum_{i=5}^8y_i^3=0,
    \end{aligned}
\end{equation}
and otherwise $\Delta(\boldsymbol{x},\boldsymbol{y},\alpha_3,\alpha_2)=0.$ 

Meanwhile, we infer by Hua's lemma \cite[Lemma 2.5]{MR1435742} that the number of solutions $1\leq \boldsymbol{x},\boldsymbol{y}\leq X$ satisfying $(\ref{hua's lemma})$ is $O(X^{10+\epsilon}).$ Therefore, by applying the triangle inequality to the right hand side of $(\ref{first eq in prop3.1})$, one deduces that
\begin{equation*}
\begin{aligned}
  \int_0^1\int_{\mathfrak{D}}\int_0^1 |F(\boldsymbol{\alpha})|^8d\boldsymbol{\alpha}&\ll \int_{\mathfrak{D}}X^{10+\epsilon}d\alpha_2d\alpha_3=X^{10+\epsilon}\cdot \text{mes}(\mathfrak{D}).
\end{aligned}
\end{equation*}
Hence, we complete the proof of Proposition $\ref{prop3.1}.$
\end{proof}

\bigskip

\begin{lemma}\label{prop3.2}
    One has
    \begin{equation}\label{claim in lemma 3.2}
\int_{\mathfrak{n}_{\delta}}|F(\boldsymbol{\alpha})|^{19}d\boldsymbol{\alpha}\ll X^{26-\delta''},
    \end{equation}
    for some $\delta''>0,$ where $d\boldsymbol{\alpha}=d\alpha_1 d\alpha_2 d\alpha_3 d\alpha_4.$
\end{lemma}
\begin{proof}
 Recall the definition of $\mathfrak{M}(H)$. We observe that for any $H\leq X^{\delta/100},$ one has
\begin{equation}\label{inclusion}
    \mathfrak{M}(H)^4\subseteq \mathfrak{N}_{\delta}.
\end{equation}

Observe that for $H>0,$ one has
\begin{equation}\label{pruning argument crucial}
\begin{aligned}
   \mathfrak{M}(H)^4 \setminus \mathfrak{M}(H/2)^4
   =\mathfrak{P}_1(H)\cup \mathfrak{P}_2(H)\cup\mathfrak{P}_3(H)\cup \mathfrak{P}_4(H),
\end{aligned}
\end{equation}
where 
\begin{equation*}
    \begin{aligned}
        \mathfrak{P}_1(H)&=\mathfrak{M}(H)^4\setminus(\mathfrak{M}(H)^3\times\mathfrak{M}(H/2))\\
 \mathfrak{P}_2(H)&=  (\mathfrak{M}(H)^3\times\mathfrak{M}(H/2))\setminus (\mathfrak{M}(H)^2\times\mathfrak{M}(H/2)^2)\\
  \mathfrak{P}_3(H)&= (\mathfrak{M}(H)^2\times\mathfrak{M}(H/2)^2)\setminus (\mathfrak{M}(H)\times\mathfrak{M}(H/2)^3)\\
  \mathfrak{P}_4(H)&=  (\mathfrak{M}(H)\times\mathfrak{M}(H/2)^3)\setminus \mathfrak{M}(H/2)^4.\end{aligned}
\end{equation*}

To verify the inequality $(\ref{claim in lemma 3.2})$, it suffices to show that for $1\leq l\leq 4$ and for all $H\geq X^{\delta/100}$ with $\delta=10^{-10}$, one has
\begin{equation}\label{claim}
\int_{ \mathfrak{P}_l(H)}|F(\boldsymbol{\alpha})|^{19}d\boldsymbol{\alpha}\ll X^{26-\delta'}\ \text{for some}\ \delta'>0.
\end{equation}
In fact, note that 
\begin{equation}\label{fdsa}
\begin{aligned}
    \mathfrak{n}_{\delta}&\subseteq \mathfrak{M}(X^{3/2})^4\setminus \mathfrak{M}(X^{\delta/100})^4\\
    &\subseteq\bigcup_{j=0}^{L}(\mathfrak{M}(2^{-j}X^{3/2})^4\setminus \mathfrak{M}(2^{-j-1}X^{3/2})^4), 
\end{aligned}
\end{equation}
for some $L=O(\log X).$
Then, it follows by $(\ref{pruning argument crucial})$ that 
\begin{equation}\label{asdf1}
\begin{aligned}
    \mathfrak{n}_{\delta}&\subseteq \bigcup_{j=0}^{L}\bigl(\mathfrak{P}_1(2^{-j}X^{3/2})\cup \mathfrak{P}_2(2^{-j}X^{3/2})\cup\mathfrak{P}_3(2^{-j}X^{3/2})\cup \mathfrak{P}_4(2^{-j}X^{3/2})\bigl).
\end{aligned}
\end{equation}
Then, we infer by $(\ref{claim})$ and $(\ref{asdf1})$  that 
\begin{equation*}
\begin{aligned}
\int_{\mathfrak{n}_{\delta}}|F(\boldsymbol{\alpha})|^{19}d\boldsymbol{\alpha}&\ll \sum_{j=0}^L\sum_{l=1}^4\int_{ \mathfrak{P}_l(2^{-j}X^{3/2})}|F(\boldsymbol{\alpha})|^{19}d\boldsymbol{\alpha} \ll X^{26-\delta''}.
\end{aligned}
\end{equation*}

Hence, we turn to verify $(\ref{claim}).$
First, we observe that for $1\leq l\leq 4$ and $H>0,$ one has
\begin{equation}\label{observation}
    \int_{ \mathfrak{P}_l(H)}|F(\boldsymbol{\alpha})|^{20}d\boldsymbol{\alpha}\ll \int_{ \mathcal{M}_l(H/2)}|F(\boldsymbol{\alpha})|^{20}d\boldsymbol{\alpha}
\end{equation}
Then, we find by $(\ref{observation})$ and Lemma $\ref{lemma2.2}$ with $s=10$ that for $1\leq l\leq 4$ one has
\begin{equation}\label{bound0}
\begin{aligned}
\int_{ \mathfrak{P}_l(H)}|F(\boldsymbol{\alpha})|^{20}d\boldsymbol{\alpha}&\ll \int_{ \mathcal{M}_l(H/2)}|F(\boldsymbol{\alpha})|^{20}d\boldsymbol{\alpha}\\
&\ll X^{8} (\log X)^{40}\cdot J_{10}(2X)\cdot (H^{-1}+X^{-1})\\
&\ll X^{28+\epsilon}(H^{-1}+X^{-1}),
\end{aligned}
\end{equation}
where we have used the sharp estimate for the cubic Parsell-Vinogradov's system \cite[Theorem 1.5]{MR3709122} with $s=10$ that 
\begin{equation*}
    J_{10}(2X)\ll X^{20+\epsilon}.
\end{equation*}

Next, by recalling the definition of $\mathfrak{P}_l(H)$,
we see by applying Lemma $\ref{prop3.1}$ with $\mathfrak{D}= \mathfrak{M}(H)^2$  that for $1\leq l\leq 4$ one has
\begin{equation}\label{bound3}
\begin{aligned}
    \int_{\mathfrak{P}_l(H)}|F(\boldsymbol{\alpha})|^{8}d\boldsymbol{\alpha}&\ll \int_0^1\int_{\mathfrak{M}(H)^2}\int_0^1|F(\boldsymbol{\alpha})|^{8}d\boldsymbol{\alpha}\\
    &\ll X^{10+\epsilon}\cdot \text{mes}(\mathfrak{M}(H)^2)\\
    &\ll X^{4+\epsilon}H^4,
\end{aligned}
\end{equation}
where we have used the fact that $\text{mes}(\mathfrak{M}(H))\leq H^2X^{-3}.$

Finally, by applying the H\"older's inequality, we deduce from $(\ref{bound0})$ and $(\ref{bound3})$ that for $1\leq l\leq 4$ one has
\begin{equation*}
\begin{aligned}
   & \int_{ \mathfrak{P}_l(H)}|F(\boldsymbol{\alpha})|^{19}d\boldsymbol{\alpha}\\
    &\leq\left(\int_{ \mathfrak{P}_l(H)}|F(\boldsymbol{\alpha})|^{20}d\boldsymbol{\alpha}\right)^{11/12}\cdot \left(\int_{ \mathfrak{P}_l(H)}|F(\boldsymbol{\alpha})|^{8}d\boldsymbol{\alpha}\right)^{1/12}\\
    &\ll (X^{28+\epsilon}(H^{-1}+X^{-1}))^{11/12}(X^{4+\epsilon}H^4)^{1/12}\\
    &\ll X^{26+\epsilon}H^{-7/12}+X^{24.75+1/3+\epsilon}H^{1/3}\\
    &\ll X^{26-\delta'}\ \ \text{for some }\delta'>0,
\end{aligned} 
\end{equation*}
where we have used the fact that $X^{\delta/100}\leq H\leq X^{3/2}.$
\end{proof}

\begin{proof}[Proof of Lemma 2.2]
    By applying the H\"older's inequality, we first deduce that
    \begin{equation}\label{holder's inequality}
        \int_{\mathfrak{n}_{\delta}} \prod_{j=1}^s F_{c_j}(\boldsymbol{\alpha})d\boldsymbol{\alpha}\ll \prod_{j=1}^s\biggl(\int_{\mathfrak{n}_{\delta}}|F_{c_j}(\boldsymbol{\alpha})|^sd\boldsymbol{\alpha}\biggr)^{1/s}.
    \end{equation}
Note that whenever $s\geq19$ one has
\begin{equation}\label{smoments}
\begin{aligned}
    \int_{\mathfrak{n}_{\delta}}|F_{c_j}(\boldsymbol{\alpha})|^sd\boldsymbol{\alpha}&\ll  X^{2s-38}\int_{\mathfrak{n}_{\delta}}|F_{c_j}(\boldsymbol{\alpha})|^{19}d\boldsymbol{\alpha}.
\end{aligned}
\end{equation}
Meanwhile, one infers that $\boldsymbol{\alpha}\in \mathfrak{n}_{\delta}$ implies $c_j\boldsymbol{\alpha}\in \mathfrak{n}_{\delta/c_j}\ (\text{mod}\ 1),$ as $\delta$ is sufficiently small. Hence, we see by change of variables that
\begin{equation}\label{almost last inequality}
\begin{aligned}
    \int_{\mathfrak{n}_{\delta}}|F_{c_j}(\boldsymbol{\alpha})|^{19}d\boldsymbol{\alpha}&=\frac{1}{c_j^4}\int_{c_j\mathfrak{n}_{\delta}}|F_1(\boldsymbol{\alpha})|^{19}d\boldsymbol{\alpha}\ll \int_{\mathfrak{n}_{\delta/c_j}}|F_1(\boldsymbol{\alpha})|^{19}d\boldsymbol{\alpha},
\end{aligned}
\end{equation}
where $c_j\mathfrak{n}_{\delta}=\{c_j\boldsymbol{\alpha}\in [0,c_j]^4:\ \boldsymbol{\alpha}\in \mathfrak{n}_{\delta}\}$. On recalling the definition $(\ref{def1.1})$ of $F(\boldsymbol{\alpha})$, we find that 
\begin{equation*}
    \begin{aligned}
F_1(\alpha_1,\alpha_2,\alpha_3,\alpha_4)&=F(\alpha_1,\alpha_2,\alpha_3,\alpha_4)+F(-\alpha_1,-\alpha_2,-\alpha_3,-\alpha_4)\\
&+F(\alpha_1,-\alpha_2,\alpha_3,-\alpha_4)+F(-\alpha_1,\alpha_2,-\alpha_3,\alpha_4)+O(X).
    \end{aligned}
\end{equation*}
By substituting this into the upper bound in $(\ref{almost last inequality}),$ one infers by the triangle inequality together with Lemma $\ref{prop3.2}$ that
\begin{equation}\label{19moments}
    \begin{aligned}
\int_{\mathfrak{n}_{\delta/c_j}}|F_1(\boldsymbol{\alpha})|^{19}d\boldsymbol{\alpha}&\ll\int_{\mathfrak{n}_{\delta/c_j}}|F(\boldsymbol{\alpha})|^{19}d\boldsymbol{\alpha}+X^{19}\\
&\ll X^{26-\eta'''},
    \end{aligned}
\end{equation}
for some $\eta'''>0,$ where we have used the observation that $(\alpha_1,\alpha_2,\alpha_3,\alpha_4)\in \mathfrak{n}_{\delta/c_j}$ implies that
\begin{equation*}
\begin{aligned}
(-\alpha_1,-\alpha_2,-\alpha_3,-\alpha_4)&\in \mathfrak{n}_{\delta/c_j}\ (\text{mod}\ 1)\\
(\alpha_1,-\alpha_2,\alpha_3,-\alpha_4)&\in \mathfrak{n}_{\delta/c_j}\ (\text{mod}\ 1)\\
(-\alpha_1,\alpha_2,-\alpha_3,\alpha_4)&\in \mathfrak{n}_{\delta/c_j}\ (\text{mod}\ 1).
\end{aligned}
\end{equation*}

Finally, by substituting $(\ref{19moments})$ into $(\ref{almost last inequality})$ and that into the last expression in $(\ref{smoments})$, we deduce that
\begin{equation}
    \int_{\mathfrak{n}_{\delta}}|F_{c_j}(\boldsymbol{\alpha})|^sd\boldsymbol{\alpha}\ll X^{2s-12-\eta''},
\end{equation}
for some $\eta''>0.$ Therefore, by substituting this into $(\ref{holder's inequality})$, this completes the proof of Lemma \ref{lemma1.3}.
\end{proof}

\section*{Acknowledgement}
The author would like to express sincere gratitude to Trevor Wooley for suggesting this problem. The author would like to thank Julia Brandes and Scott Parsell for many helpful comments and encouragement. Especially, the author would like to thank Scott Parsell for pointing out an issue in the earlier draft and providing suggestions to resolve this issue. The author also gratefully acknowledges support from the KAP allocation at the University of California, Davis.


\bibliographystyle{alpha}
\bibliography{reference}

\end{document}